\documentclass{amsart}
\usepackage{amssymb}
\usepackage{amsfonts}

\newtheorem{theorem}{Theorem}
\theoremstyle{plain}

\newtheorem{corollary}{Corollary}

\newtheorem{definition}{Definition}

\newtheorem{remark}{Remark}

\numberwithin{equation}{section}
\input{tcilatex}

\begin{document}
\title[Homogeneous Beta-type functions]{Homogeneous Beta-type functions\\
}
\author{Martin Himmel}
\curraddr{Faculty of Mathematics, Computer Science and Econometrics
University of Zielona G\'{o}ra, Szafrana 4A, PL 65-516 Zielona G\'{o}ra,
Poland}
\email{himmel@mathematik.uni-mainz.de}
\author{Janusz Matkowski}
\curraddr{Faculty of Mathematics, Computer Science and Econometrics
University of Zielona G\'{o}ra, Szafrana 4A, PL 65-516 Zielona G\'{o}ra,
Poland}
\email{J.Matkowski@wmie.uz.zgora.pl }

\begin{abstract}
All beta-type functions, i.e. the functions $B_{f}:\left( 0,\infty \right)
^{2}\rightarrow \left( 0,\infty \right) $ of the form%
\begin{equation*}
B_{f}\left( x,y\right) =\frac{f\left( x\right) f\left( y\right) }{f\left(
x+y\right) }
\end{equation*}%
for some $f:\left( 0,\infty \right) \rightarrow \left( 0,\infty \right) ,$
which are $p$-homogeneous, are determined. Applying this result, we show
that a beta-type function is a homogeneous mean iff it is the harmonic one.
A reformulation of a result due to Heuvers in terms of a Cauchy difference
and the harmonic mean is given.
\end{abstract}

\maketitle

\QTP{Body Math}
$\bigskip $

\section{Introduction}

For a given $f:\left( 0,\infty \right) \rightarrow \left( 0,\infty \right) ,$
the function $B_{f}:\left( 0,\infty \right) ^{2}\rightarrow \left( 0,\infty
\right) $ defined by%
\begin{equation*}
B_{f}\left( x,y\right) =\frac{f\left( x\right) f\left( y\right) }{f\left(
x+y\right) },\text{ \ \ \ \ \ \ }x,y>0,
\end{equation*}%
is called the \textit{beta-type function,} and $f$ is called its \textit{%
generator} (\cite{MatHim}). The notion the beta-type function arises from
the well-known relation between the Euler Beta function $B:\left( 0,\infty
\right) ^{2}\rightarrow \left( 0,\infty \right) $ and the Euler Gamma
function $\Gamma :$ $\left( 0,\infty \right) \rightarrow \left( 0,\infty
\right) $ 
\begin{equation*}
B\left( x,y\right) =\frac{\Gamma \left( x\right) \Gamma \left( y\right) }{%
\Gamma \left( x+y\right) },\text{ \ \ \ \ \ }x,y>0.
\end{equation*}%
Given $p\in \mathbb{R}$, we examine when the beta-type function $B_{f}$ is $%
p $-homogeneous, i.e. when%
\begin{equation*}
B_{f}\left( tx,ty\right) =t^{p}B_{f}\left( x,y\right) ,\text{ \ \ \ \ \ }%
x,y>0.
\end{equation*}%
Theorem \ref{theo:HomogeneityBetaType}, the main result, says that, under
some regularity assumptions of the generator $f$, the beta-type function is $%
p$-homogeneous if, and only if, there exist $a,b>0$ such that $f\left(
x\right) =bxa^{x}$ for all $x>0$. As a corollary we obtain that a beta-type
function is a homogeneous pre-mean if, and only if, there exists $a>0$ such
that $f\left( x\right) =2xa^{x}$ for all $x>0$, or, equivalently, that $%
B_{f} $ is the harmonic mean, that is $B_{f}=H,$ where 
\begin{equation*}
H\left( x,y\right) =\frac{2xy}{x+y},\text{ \ \ \ \ \ }x,y>0.
\end{equation*}

A related companion of the beta-type function is the Cauchy difference $%
C_{g}:\left( 0,\infty \right) ^{2}\rightarrow \mathbb{R}$ defined by%
\begin{equation*}
C_{g}\left( x,y\right) =g\left( x+y\right) -g\left( x\right) -g\left(
y\right)
\end{equation*}%
for a function $g:\left( 0,\infty \right) \rightarrow \mathbb{R}$. The
relationship 
\begin{equation*}
B_{f}=\exp \circ \left( -C_{\log \circ f}\right)
\end{equation*}%
allows to reformulate Theorem 1 in terms of logarithmical homogeneity of the
Cauchy difference (Corollary 3).

At the end we remark that Heuvers result \cite{Heuvers} on a
characterization of logarithmic functions can be reformulated in terms of
the Cauchy difference and the harmonic mean.

\section{Main result}

\begin{theorem}
\label{theo:HomogeneityBetaType} Let a function $f:\left( 0,\infty \right)
\rightarrow \left( 0,\infty \right) $ be continuous or Lebesgue measurable.
Then the following conditions are equivalent:

(i) the beta-type function $B_{f}$ is $p$-homogeneous, i.e.%
\begin{equation*}
B_{f}\left( tx,ty\right) =t^{p}B_{f}\left( x,y\right) ,\text{ \ \ \ \ \ \ }%
x,y,t>0;
\end{equation*}

(ii) there exist $a,b\in \left( 0,\infty \right) $ such that%
\begin{equation*}
f\left( x\right) =bxa^{x},\ \ \ \ x>0
\end{equation*}%
and 
\begin{equation*}
B_{f}\left( x,y\right) =b\left( \frac{xy}{x+y}\right) ^{p},\text{ \ \ \ \ \ }%
x,y>0.
\end{equation*}
\end{theorem}

\begin{proof}
Assume $(i)$ holds. Hence, by the definition of \ $B_{f}$, we have 
\begin{equation}
\frac{f\left( tx\right) f\left( ty\right) }{f\left( t\left( x+y\right)
\right) }=t^{p}\frac{f\left( x\right) f\left( y\right) }{f\left( x+y\right) }%
,\text{ \ \ \ \ \ }x,y,t>0,  \label{HomBTF1}
\end{equation}%
which can be written in the form%
\begin{equation}
\frac{f\left( t\left( x+y\right) \right) }{t^{p}f\left( x+y\right) }=\frac{%
f\left( tx\right) }{t^{p}f\left( x\right) }\frac{f\left( ty\right) }{%
t^{p}f\left( y\right) },\text{ \ \ \ \ \ }x,y,t>0.  \label{HomBTF2}
\end{equation}%
For every fixed $t>0$ define $\varphi _{t}:\left( 0,\infty \right)
\rightarrow \left( 0,\infty \right) $ by%
\begin{equation*}
\varphi _{t}\left( x\right) :=\frac{f\left( tx\right) }{t^{p}f\left(
x\right) },\text{ \ \ \ \ \ \ }x>0.
\end{equation*}%
Thus, from (\ref{HomBTF2}), for arbitrary fixed $t>0,$ it holds%
\begin{equation*}
\varphi _{t}\left( x+y\right) =\varphi _{t}\left( x\right) \varphi
_{t}\left( y\right) ,\text{ \ \ \ \ \ }x,y>0,
\end{equation*}%
stating that $\varphi _{t}$ is an exponential function. Hence (see, for
instance, \cite{Aczel} p. 39), for every $t>0,$ there exists a unique
additive function $\alpha _{t}:\mathbb{R}\rightarrow \mathbb{R}$ such that 
\begin{equation*}
\varphi _{t}\left( x\right) =e^{\alpha _{t}\left( x\right) },\text{ \ \ \ \
\ }x>0.
\end{equation*}%
From the definition of $\varphi _{t}$, we have%
\begin{equation*}
e^{\alpha _{t}\left( x\right) }t^{p}f\left( x\right) =f\left( tx\right) ,%
\text{ \ \ \ \ \ }x>0.
\end{equation*}%
Since the right hand side is symmetric in $x$ and $t,$ so is the left hand
side; thus 
\begin{equation*}
e^{\alpha _{x}\left( t\right) }x^{p}f\left( t\right) =f\left( xt\right)
=f\left( tx\right) =e^{\alpha _{t}\left( x\right) }t^{p}f\left( x\right) ,%
\text{ \ \ \ \ \ }x,t>0.
\end{equation*}%
Setting here $t=1$ gives%
\begin{equation*}
e^{\alpha _{1}\left( x\right) }f\left( x\right) =f\left( x\right) =e^{\alpha
_{x}\left( 1\right) }x^{p}f\left( 1\right) ,\text{ \ \ \ \ \ }x>0,
\end{equation*}%
and as, by assumption, $f$ is positive, it follows that 
\begin{equation*}
\alpha _{1}\left( x\right) =0,\text{ \ \ \ \ \ }x>0.
\end{equation*}%
and, consequently, 
\begin{equation*}
f\left( x\right) =f\left( 1\right) x^{p}e^{\alpha _{x}\left( 1\right) },%
\text{ \ \ \ \ \ }x>0.
\end{equation*}%
Putting, for convenience, $\lambda :\left( 0,\infty \right) \mathbb{%
\rightarrow R}$,%
\begin{equation*}
\lambda \left( x\right) :=\alpha _{x}\left( 1\right) ,\text{ \ \ \ \ \ }x>0,
\end{equation*}%
we have%
\begin{equation}
f\left( x\right) =f\left( 1\right) x^{p}e^{\lambda \left( x\right) },\text{
\ \ \ \ \ }x>0.  \label{canGenHomBF}
\end{equation}%
Inserting this into (\ref{HomBTF1}), we obtain, 
\begin{equation*}
\frac{f\left( 1\right) \left( tx\right) ^{p}e^{\lambda \left( tx\right)
}f\left( 1\right) \left( ty\right) ^{p}e^{\lambda \left( ty\right) }}{%
f\left( 1\right) \left[ t\left( x+y\right) \right] ^{p}e^{\lambda \left(
t\left( x+y\right) \right) }}=t^{p}\frac{f\left( 1\right) x^{p}e^{\lambda
\left( x\right) }f\left( 1\right) y^{p}e^{\lambda \left( y\right) }}{f\left(
1\right) \left( x+y\right) ^{p}e^{\lambda \left( x+y\right) }},\text{ \ \ \
\ \ }x,y,t>0,
\end{equation*}%
that reduces to%
\begin{equation*}
e^{\lambda \left( tx\right) +\lambda \left( ty\right) -\lambda \left(
t\left( x+y\right) \right) }=e^{\lambda \left( x\right) +\lambda \left(
y\right) -\lambda \left( x+y\right) },\text{ \ \ \ \ \ }x,y,t>0,
\end{equation*}%
whence 
\begin{equation*}
\lambda \left( tx\right) +\lambda \left( ty\right) -\lambda \left( t\left(
x+y\right) \right) =\lambda \left( x\right) +\lambda \left( y\right)
-\lambda \left( x+y\right) ,\text{ \ \ \ \ \ }x,y,t>0.
\end{equation*}%
Writing this in the form%
\begin{equation*}
\lambda \left( t\left( x+y\right) \right) -\lambda \left( x+y\right) =\left[
\lambda \left( tx\right) -\lambda \left( x\right) \right] +\left[ \lambda
\left( ty\right) -\lambda \left( y\right) \right] ,\text{ \ \ \ \ \ }x,y,t>0,
\end{equation*}%
we conclude that, for any $t>0$, the function $\omega =\omega _{t}:\left(
0,\infty \right) \rightarrow \mathbb{R}$, defined by 
\begin{equation}
\omega \left( x\right) :=\lambda \left( tx\right) -\lambda \left( x\right) ,%
\text{ \ \ \ \ \ }x>0,  \label{Lambda(tx)-Lambda(x)}
\end{equation}%
is additive. From (\ref{canGenHomBF}) and the assumed regularity of $f$ we
get that $\omega $ is continuous or Lebesgue measurable. Thus, $\omega ,$
being additive and continuous or measurable, is of the form (\cite{Kuczma2},
p. 129, see also \cite{Aczel})%
\begin{equation*}
\omega \left( x\right) =\omega \left( 1\right) x,\text{ \ \ \ \ \ }x>0,
\end{equation*}%
and hence, by (\ref{Lambda(tx)-Lambda(x)}), 
\begin{equation*}
\lambda \left( tx\right) -\lambda \left( x\right) =\left( \lambda \left(
t\right) -\lambda \left( 1\right) \right) x,\text{ \ \ \ \ \ }x,t>0,
\end{equation*}%
whence%
\begin{equation*}
\lambda \left( tx\right) =\lambda \left( x\right) +\left( \lambda \left(
t\right) -\lambda \left( 1\right) \right) x,\text{ \ \ \ \ \ }x,t>0.
\end{equation*}%
The symmetry in $t$ and $x$\ of the left hand side implies that%
\begin{equation*}
\lambda \left( x\right) +\left( \lambda \left( t\right) -\lambda \left(
1\right) \right) x=\lambda \left( t\right) +\left( \lambda \left( x\right)
-\lambda \left( 1\right) \right) t,\text{ \ \ \ \ \ \ \ \ }x,t>0,
\end{equation*}%
whence%
\begin{equation*}
\lambda \left( x\right) \left( 1-t\right) +\lambda \left( 1\right) t=\lambda
\left( t\right) \left( 1-x\right) +\lambda \left( 1\right) x,\text{ \ \ \ \
\ \ \ }x,t>0.
\end{equation*}%
Subtracting $\lambda \left( 1\right) $ from both sides yields%
\begin{equation*}
\lambda \left( x\right) \left( 1-t\right) +\lambda \left( 1\right) t-\lambda
\left( 1\right) =\lambda \left( t\right) \left( 1-x\right) +\lambda \left(
1\right) x-\lambda \left( 1\right) ,\text{ \ \ \ \ \ \ \ }x,t>0,
\end{equation*}%
whence%
\begin{equation*}
\lambda \left( x\right) \left( 1-t\right) -\lambda \left( 1\right) \left(
1-t\right) =\lambda \left( t\right) \left( 1-x\right) -\lambda \left(
1\right) \left( 1-x\right) ,\text{ \ \ \ \ \ \ }x,t>0,
\end{equation*}%
and, consequently, 
\begin{equation*}
\frac{\lambda \left( x\right) -\lambda \left( 1\right) }{1-x}=\frac{\lambda
\left( t\right) -\lambda \left( 1\right) }{1-t},\text{ \ \ \ \ \ \ }x,t>0,%
\text{ }x\neq 1\neq y.
\end{equation*}%
It follows that there exists $c\in \mathbb{R}$ such that%
\begin{equation*}
\frac{\lambda \left( x\right) -\lambda \left( 1\right) }{1-x}=-c,\text{ \ \
\ \ \ \ }x>0,x\neq 1,
\end{equation*}%
whence, 
\begin{equation*}
\lambda \left( x\right) =c\left( x-1\right) +\lambda \left( 1\right) ,\text{
\ \ \ \ \ }x>0,
\end{equation*}%
and we obtain%
\begin{equation*}
\lambda \left( x\right) =cx+d,\text{ \ \ \ \ \ }x>0,
\end{equation*}%
where $d:=\lambda \left( 1\right) -c$. Inserting this function $\lambda $
into (\ref{canGenHomBF}), we obtain 
\begin{equation*}
f\left( x\right) =f\left( 1\right) e^{d}x^{p}\left( e^{c}\right) ^{x},\text{
\ \ \ \ \ }x>0,
\end{equation*}%
whence, setting 
\begin{equation*}
a:=e^{c},\text{ \ \ \ \ \ }b:=f\left( 1\right) e^{d},
\end{equation*}%
we get 
\begin{equation*}
f\left( x\right) =bx^{p}a^{x},\text{ \ \ \ \ \ }x>0,
\end{equation*}%
and 
\begin{equation*}
B_{f}\left( x,y\right) =b\left( \frac{xy}{x+y}\right) ^{p},\text{ \ \ \ \ \ }%
x,y>0,
\end{equation*}%
which proves $(ii)$. The implication $(ii)\Longrightarrow (i)$ is obvious.
\end{proof}

\section{Applications to pre-means}

\begin{definition}
Let $I\subseteq \mathbb{R}$ be an interval and $M:I^{2}\rightarrow \mathbb{R}
$. The $M$ is reflexive, if%
\begin{equation*}
M\left( x,x\right) =x,\ \ \ \ \ x\in I;
\end{equation*}

$M$ is called a \textit{pre-mean in }$I$ (\cite{MatNov2014}), if it is
reflexive and $M\left( I^{2}\right) \subseteq I;$

$M$ is called a \textit{mean in }$I,$ \textit{if }%
\begin{equation*}
\min \left( x,y\right) \leq M\left( x,y\right) \leq \max \left( x,y\right)
,\ \ \ \ \ \ x,y\in I.
\end{equation*}
\end{definition}

\begin{remark}
If $M:I^{2}\rightarrow \mathbb{R}$ is reflexive, then $I\subseteq M\left(
I^{2}\right) $; so a reflexive function is a pre-mean if, and only if, $%
M\left( I^{2}\right) =I.$
\end{remark}

\begin{remark}
Obviously, every mean is a pre-mean, but, in general, not vice versa.
Indeed, the function $M:\left( 0,\infty \right) ^{2}\rightarrow \left(
0,\infty \right) $ defined by%
\begin{equation*}
M\left( x,y\right) =\frac{2x^{2}+y^{2}}{x+2y}
\end{equation*}%
is a pre-mean. Since $M\left( 2,1\right) =3\notin \left[ 2,1\right] $ the
function is not a mean. So $M$ is not increasing in both variables because,
otherwise, it would be a mean.
\end{remark}

\begin{remark}
\label{rem:HomRef} If $M:\left( 0,\infty \right) ^{2}\rightarrow \mathbb{R}$
is reflexive and, for some $p\in \mathbb{R},$ $p$-homogenous, then $p=1$.
\end{remark}

\begin{corollary}
\label{corr:HomogeneityBetaType} Let $f:\left( 0,\infty \right) \rightarrow
\left( 0,\infty \right) $ be a continuous function. Then the following
conditions are equivalent:

(i) the beta-type function $B_{f}$ is a homogeneous pre-mean;

(ii) there exists $a\in \left( 0,\infty \right) $ such that%
\begin{equation}
f\left( x\right) =2xa^{x},\ \ \ \ x>0;  \label{eq:gerHarMean}
\end{equation}

(iii) the beta-type function coincides with the harmonic mean, i.e. 
\begin{equation*}
B_{f}\left( x,y\right) =\frac{2xy}{x+y},\text{ \ \ \ \ \ }x,y>0.
\end{equation*}
\end{corollary}

\begin{proof}
Assume $(i)$. By Theorem \ref{theo:HomogeneityBetaType} and remark \ref%
{rem:HomRef}, its generator $f$ is of the form%
\begin{equation*}
f\left( x\right) =bxa^{x},\ \ \ \ x>0,
\end{equation*}%
for some $a,b\in \left( 0,\infty \right) $. Since $B_{f}$ is reflexive, that
is $B_{f}\left( x,x\right) =x$ for all $x\in \left( 0,\infty \right) .$
Substituting here $x=2$ and using Theorem \ref{theo:HomogeneityBetaType}
(ii), yields 
\begin{equation*}
2=B_{f}\left( 2,2\right) =\frac{f\left( 2\right) f\left( 2\right) }{f\left(
2+2\right) }=\frac{b\cdot 2\cdot 2}{2+2}=b,
\end{equation*}%
whence we get $\left( \ref{eq:gerHarMean}\right) ,$ which proves (ii).

Assume $(ii).$ From $\left( \ref{eq:gerHarMean}\right) $ and the definition
of $B_{f}$ we get (iii).

The implication $(iii)\Longrightarrow (i)$ is obvious.
\end{proof}

Because every homogeneous quasi-arithmetic mean is a power mean (\cite%
{AcyDho}, p. 249), our result implies the following

\begin{corollary}
A homogeneous beta-type function is a quasi-arithmetic mean if, and only if,
it is the harmonic mean.
\end{corollary}

For another result connecting harmonic mean and the Euler Gamma function see 
\cite{Alzer1996}.

\section{Cauchy differences and a corollary}

Applying our main result, we obtain the following

\begin{corollary}
Let $g:\left( 0,\infty \right) \rightarrow \mathbb{R}$ be an arbitrary
continuous function and let $p\in \mathbb{R}$. The following conditions are
equivalent:

(i) the Cauchy difference is $p\log t$-homogeneous, that is%
\begin{equation}
C_{g}\left( tx,ty\right) =C_{g}\left( x,y\right) +p\log t,\text{ \ \ \ \ \ }%
x,y,t>0;  \label{plogHom}
\end{equation}%
(ii) there exist $c,d\in \mathbb{R}$ such that 
\begin{equation*}
g\left( x\right) =cx+d-p\log t,\text{ \ \ \ \ \ }x>0
\end{equation*}%
and%
\begin{equation*}
C_{g}\left( x,y\right) =\log \left( \frac{xy}{x+y}\right) ^{p}-d,\text{ \ \
\ \ \ }x,y>0.
\end{equation*}
\end{corollary}

\begin{proof}
Setting $f:=\exp \circ g,$ we observe that condition $(i)$ is equivalent to%
\begin{equation*}
B_{f}\left( tx,ty\right) =t^{-p}B_{f}\left( x,y\right) ,\text{ \ \ \ \ }%
x,y,t>0,
\end{equation*}%
since, using the definition of beta-type function, we have, for all $x,y>0,$%
\begin{equation*}
e^{g\left( tx\right) +g\left( ty\right) -g\left( t\left( x+y\right) \right)
}=t^{-p}e^{g\left( x\right) +g\left( y\right) -g\left( x+y\right) }.
\end{equation*}%
Taking the logarithm of both sides, we indeed obtain%
\begin{equation*}
-C_{g}\left( tx,ty\right) =\log t^{-p}-C_{g}\left( x,y\right) ,\text{ \ \ \
\ }x,y>0,
\end{equation*}%
and thus $g$ satisfies $\left( \ref{plogHom}\right) $.

By Theorem \ref{theo:HomogeneityBetaType}, there exist $a,b>0$%
\begin{equation*}
f\left( x\right) =bx^{-p}a^{x},\text{ \ \ \ \ }x>0.
\end{equation*}%
Thus, by the definition of $f$, we get, for all $x>0,$%
\begin{equation*}
g\left( x\right) =\log b+p\log x+x\log a;
\end{equation*}%
whence, putting $c:=\log a$ and $d:=\log b,$ we obtain,%
\begin{equation*}
g\left( x\right) =cx+d+p\log x,\text{ \ \ \ \ }x>0,
\end{equation*}%
and consequently, for all $x,y>0$,%
\begin{eqnarray*}
C_{g}\left( x,y\right) &=&g\left( x+y\right) -g\left( x\right) -g\left(
y\right) \\
&=&\log \left( \frac{xy}{x+y}\right) ^{p}-d,
\end{eqnarray*}%
which proves the implication $(i)\Longrightarrow (ii).$

The second implication is easy to verify.
\end{proof}

In connection with Cauchy differences and harmonic mean, let us note that
Heuvers result \cite{Heuvers} (see also Kannappan \cite{Kannappan}, p. 31)
can be reformulated as

\begin{remark}
\label{rem:CauchyDifferenceHarmonicMean} The Cauchy difference of a function 
$f:\left( 0,\infty \right) \rightarrow \mathbb{R}$ satisfies the functional
equation%
\begin{equation}
C_{f}\left( x,y\right) =f\left( \frac{2}{H\left( x,y\right) }\right) ,\text{
\ \ \ \ }x,y>0  \label{LogHeuvers}
\end{equation}%
if, and only if, $f$ is a logarithmic function, i.e.%
\begin{equation*}
f\left( xy\right) =f\left( x\right) +f\left( y\right) ,\text{ \ \ \ \ \ }%
x,y>0.
\end{equation*}
\end{remark}

\bigskip

\end{document}